\documentclass[10pt]{article}

\usepackage{xcolor} 
\newcommand{\Cr}[1]{\textbf{\textcolor{red}{{#1}}}}

\newcommand{\Cp}[1]{\textcolor{violet}{{#1}}}

\usepackage{enumerate}

\usepackage{hyperref}

\usepackage{amsmath}
\usepackage{amsthm}
\usepackage{amssymb}
\usepackage{amsfonts}
\usepackage{bbm}
\usepackage{nicefrac}
\usepackage{mathtools}
\usepackage{soul}

\usepackage[title, toc]{appendix}

\usepackage{tikz}
\usepackage{color}
\usepackage{graphicx}
\usepackage{fullpage}
\usepackage{float}
\usepackage{wrapfig}
\usepackage{caption}
\usepackage{subcaption}
\usepackage{listings}
\usepackage{verbatim}
\numberwithin{equation}{section}

\let\Horig\H

\DeclareMathOperator{\Var}{Var}

\DeclareMathOperator*{\re}{Re}

\newcommand{\beq}{ \begin{equation} }
\newcommand{\eeq}{ \end{equation} }
\newcommand{\beqq}{ \begin{equation*} }
\newcommand{\eeqq}{ \end{equation*} }

\newcommand{\ii}{\mathrm{i}}

\newcommand{\RN}[1]{%
  \textup{\uppercase\expandafter{\romannumeral#1}}%
}

\def \simeqids {\stackrel{\mathcal{D}}\simeq}

\def \P {\mathbb{P}}
\def \R {\mathbb{R}}

\def \dd {\mathrm{d}}

\def \g {\gamma}
\def \e {\delta} 

\def \l {\lambda}

\newcommand{\rma}[1]{\mathcal{#1}}
\newcommand{\gib}[1]{\mathfrak{#1}}

\def \ham {\rma H}
\def \tmp {T}

\def \sGOE {M}
\def \eg {\lambda}

\def \pat {\rma Z}

\def \ef {h}
\def \efv {\mathbf{g}}
\def \efve {g}
\def \efres {H}

\def \cp {\gamma}
\def \scl{\dd \sigma_{scl}(x)}

\def \airy {\alpha}

\def \egres {a}
\def \G {\rma G}
\def \event {\mathcal{E}_{\varepsilon}}
\def \ep {\varepsilon}

\def \mgn {\gib M}
\def \scp {\rma X}
\def \Gmgn {\rma G_{\mgn}}
\def \cpmgn {\cp_{\mgn}}

\def \cF {\mathcal{F}}

\def \PP {\mathbb{P}}

\newtheorem{theorem}{Theorem}[section]

\newtheorem{lemma}[theorem]{Lemma}

\theoremstyle{remark}

\DeclareMathOperator{\sphve}{\sigma}

\def \ovl {\gib R}

\def \so {p} 

\def \spin {\boldsymbol{\sigma}}
\def \sphv {\spin}

\begin{document}

\title{Overlap of a spherical spin glass model with microscopic external field}
\author{
Elizabeth Collins-Woodfin\footnote{Department of Mathematics, University of Michigan,
Ann Arbor, MI, 48109, USA \newline email: \texttt{elicolli@umich.edu}} 
}

\maketitle

\begin{abstract}
We examine the behavior of the 2-spin spherical Sherrington-Kirkpatrick model with an external field by analyzing the overlap of a spin with the external field.  Previous research has noted that, at low temperature, this overlap exhibits dramatically different behavior in the presence of an external field as compared to the model with no external field.  The transition between those two settings was examined in a recent physics paper by Baik, Collins-Woodfin, Le Doussal, and Wu as well as a recent math paper by Landon and Sosoe. Both papers focus on the setting in which the external field strength, $h$, approaches zero as the dimension, $N$, approaches infinity.  
In particular, the paper of Baik et al studies the overlap with a microscopic external field ($h\sim N^{-1/2}$) but without a rigorous proof. This paper aims to give a proof of that result.
The proof involves representing the generating function of the overlap as a ratio of contour integrals and then analyzing the asymptotics of those contour integrals using results from random matrix theory.

\end{abstract}


\section{Introduction}
\subsection{Model set-up and notations}
The 2-spin spherical Sherrington-Kirkpatrick (SSK) spin glass model involves a spin variable $\spin=(\sigma_1, \cdots, \sigma_N)$ in $S_{N-1}$, the sphere of radius $\sqrt{N}$ in $\R^N$:
\beqq
	S_{N-1}=\{ \spin \in \R^N : \|\spin\|=\sqrt{N}\}.
\eeqq
The SSK model with external field is defined by the Hamiltonian 
\beq
\label{eq:def_Hamilt}
\ham(\sphv) =  - \frac12 \sum_{i,j = 1}^{N} \sGOE_{ij} \sphve_i \sphve_j - \ef \sum_{i = 1}^N \efve_i \sphve_i 
	= - \frac 12\sphv \cdot \sGOE \sphv  - h  \, \efv\cdot \sphv 
\eeq 
where $\efv$ is a standard gaussian random vector and $M$ is an $N\times N$ random symmetric matrix from the Gaussian orthogonal ensemble (GOE).  More specifically, for $i\le j$, the variables $\sGOE_{ij}$ are independent centered Gaussian random variables with variance $\frac1N$ for $i<j$ and $\frac2N$ for $i=j$.
By the symmetry condition,  $\sGOE_{ij}= \sGOE_{ji}$ for $i>j$. 

The Gibbs measure for this model is 
\beq\label{eq:Gibbsmeasure}
	p(\sphv)= \frac1{\pat_N} e^{-\beta \ham(\spin)} \quad \text{for $\sphv\in S_{N-1}$}
\eeq
where the parameter $\beta>0$ denotes the inverse temperature and is also written $\beta=\frac1T$. 
The partition function  $\pat_N$ is given by
\beq
\label{eq:def_fe}
	\pat_N = \int_{S_{N - 1}} e^{-\beta \ham(\spin)} \dd \omega_N(\sphv)
\eeq 
where $\omega_N$ is the normalized uniform measure on $S_{N - 1}$. Since $\sGOE$ and $\efv$ are random, the Gibbs measure is a random measure.   We use the notation
\beq
\langle A\rangle=\int A(\spin)\dd p(\spin)
\eeq
to denote the expectation of $A$ with respect to the Gibbs measure where $A$ depends on $\spin$.  Since the Gibbs measure depends on $M$ and $\mathbf{g}$, the Gibbs expectation $\langle A \rangle$ is a function of $M$ and $\mathbf{g}$.

Overlaps are of particular interest in the study of SSK and other spin glass models.  In this paper we focus on the overlap of a spin with the external field, denoted by $\mgn$, and briefly discuss the overlap with a replica, denoted by $\ovl$. These are defined as
\beq \label{eq:magntde}
	\mgn =\frac{\efv \cdot \sphv}{N}
	\quad\text{and}\quad
	\mathfrak{R}=\frac{\spin^{(1)} \cdot\spin^{(2)}}N
\eeq
where $\sphv$ is chosen randomly according to the Gibbs measure
and $\spin^{(1)},\spin^{(2)}$ are two independent copies (or replicas) of $\spin$, chosen from the Gibbs measure with the same values for $M$ and $\efv$.

We denote the eigenvalues of $M$ and their corresponding unit eigenvectors by
\beqq
\lambda_1\geq\lambda_2\geq\cdots\geq\lambda_N\qquad\text{and}\qquad\mathbf{u}_1,\mathbf{u}_2,...,\mathbf{u}_N.
\eeqq
Many calculations will involve the inner product of the external field with an eigenvector.  For this purpose, we introduce the notation
\beqq
n_i=\efv\cdot\mathbf{u}_i.
\eeqq

\subsection{Background}
The SSK model was first studied by Kosterlitz, Thouless and Jones \cite{kosterlitz1976spherical} as a continuous analog of the model with Ising spins ($\spin\in\{-1,+1\}^N$) that was introduced by Sherrington and Kirkpatrick \cite{SherringtonKirkpatrick1975}.  Some of the most influential early work on these models focused on calculating the limiting  free energy (as $N\to\infty$) for the SK model, SSK model and the more general $p$-spin models.  A formula for the limiting free energy of the SK model (as well as the more general $p$-spin case) was first established by Parisi \cite{parisi1980sequence}, and analogous formulas for spherical models were developed by Kosterlitz et al \cite{kosterlitz1976spherical} for the SSK model and by Crisanti and Sommers \cite{crisanti1992sphericalp} for $p$-spin spherical models.  These formulas were later proven rigorously by Talagrand \cite{TalagrandSK,TalagrandSSK}.

The SSK model allows for certain the types of analysis that are not possible with the SK model.  In particular, the partition function for SSK has an equivalent formulation as a contour integral, which we describe in more detail in Section \ref{sec:integralreps}.  This contour integral representation of $\pat_N$ was first observed by \cite{kosterlitz1976spherical} and was later used by Baik and Lee \cite{BaikLeeSSK} to analyze the free energy of SSK with zero external field up to fluctuations for both the high temperature ($T>1$) and low temperature ($T<1$) cases.  They also used this method to analyze the free energy of several related models \cite{BaikLeeFerromagnetic, BaikLeeBipartite, BaikLeeWu}.

Overlaps are useful in terms of studying the distribution of spins (see, for example, \cite{Subag_2017} for a discussion of the geometry of the Gibbs measure in $p$-spin spherical models without external field).  In the case of SSK with $h=0$, the overlap with a replica concentrates around the values $\pm(1-T)$ \cite{panchenko2007overlap}.  The contour integral representation of the partition function can be used to study the fluctuations of this overlap.  In particular, Nguyen and Sosoe \cite{ nguyen2018central} find that, at high temperature and zero external field, the overlap with a replica converges to a centered gaussian distribution.  Landon and Sosoe \cite{ LandonSosoe} extend this analysis to the low temperature case and find that the overlap with a replica is no longer normally distributed but instead converges to a mean zero, bimodal random variable whose distribution can be expressed explicitly as a function of the GOE Airy point process.

The aforementioned results all indicate that, in the $h=0$ setting, the SSK model displays a phase transition at temperature $T=1$.  However, this transition does not occur in the case where we have fixed positive external field strength $h>0$ \cite{kosterlitz1976spherical, crisanti1992sphericalp, TalagrandSSK, ChenSen}.  In that case, the overlap concentrates around a positive value, $q(h)$, for all temperatures \cite{panchenko2007overlap}.

Since the low temperature case reveals a striking contrast between the SSK model with $h=0$ and the one with $h>0$, it is interesting to consider the behavior of the low temperature model when $h\to0$ as $N\to\infty$.  This question has been considered in the physics context by Fyodorov and le Doussal \cite{FyodorovleDoussal} who examine the free energy of SSK with external field and find transitional regimes at scalings $h\sim N^{-1/6}$ and $h\sim N^{-1/2}$ in the zero temperature case (see also \cite{DemboZeitouni, kivimae2019critical}).  We refer to these scalings as mesoscopic and microscopic respectively.

Recent papers by Landon and Sosoe \cite{LS} and by Baik, Collins-Woodfin, le Doussal, and Wu \cite{BCLW} have computed the distribution of the overlap with the external field as well as the overlap with a replica at the transitional scalings $h\sim N^{-1/6}$ and $h\sim N^{-1/2}$ for $T<1$.  (Both papers also analyze free energy, 
but that will not be discussed in the current paper).  At the microscopic scaling, both papers find that the overlap with a replica concentrates around the values $\pm(1-T)$, as it does in the $h=0$ case.  However, the mean is no longer zero, but positive, and the mean increases as $hN^{1/2}$ increases.   We note that both papers use a similar method of representing the overlap in terms of contour integrals. The difference is that, while \cite{LS} is mathematically rigorous, \cite{BCLW} does not provide all the details of the argument but instead focuses on the physical implications.

In addition to the overlap with a replica, it is also interesting to study the overlap with the external field.  This quantity is used by physicists to study magnetism and susceptibility (for example \cite{kosterlitz1976spherical,cugliandolo2007nonlinear}), which will be discussed further in Section \ref{sec:scp}.  Both \cite{LS} and \cite{BCLW} analyze the overlap with the external field at the macroscopic ($h=O(1)$) and mesoscopic scalings. The paper \cite{BCLW} also analyzes this overlap at the microscopic scaling by providing a non-rigorous computation of the moment generating function for the overlap.  Their computation suggests that, in its leading order, the overlap is distributed like the sum of two independent random variables, one of which is Bernoulli and one of which is Gaussian, and both are of order $N^{-1/2}$.  Providing a rigorous proof of this result will be the focus of the current paper.  
\subsection{Scope and organization of this paper}

As mentioned above, the goal of the current paper is to provide a rigorous proof of the formula conjectured in \cite{BCLW} for the distribution of the overlap with the external field at the microscopic scaling.  Although some intermediate steps of \cite{BCLW} are rigorous, many details are omitted and, most notably, they do not provide rigorous proofs for the asymptotics of the integrals.  We follow similar steps to those used in \cite{BCLW}, but fill in the missing details and supply the rigorous asymptotic analysis of the integrals. Furthermore, we provide more specificity regarding the probability with which the results hold.  The authors of \cite{BCLW} conjecture that the formula holds with a probability that tends to 1 as $N\to\infty$.  We show that it holds with probability at least $1-N^{-\ep/10}$ for any sufficiently small $\ep>0$.  

The proof in this paper utilizes the contour integral representation of the partition function as well as many results from random matrix theory.  
The asymptotic evaluation of the contour integral associated with the overlap requires a choice of the contour. The papers \cite{LS} and \cite{BCLW} use the steepest-descent contour. We found it simpler to use an explicit contour that agrees with the true steepest-descent contour only locally.

Finally, we note that our method can also be used to prove the moment generating function for the overlap with a replica, as conjectured in \cite{BCLW}.  However, we do not provide the details of that proof, since the result is proved via a different method in \cite{LS}.

Section \ref{sec:results} summarizes the main theorem to be proved in this paper and its implications.  In Section \ref{sec:prelim} we provide the notations and prerequisites that we will use throughout the paper including some lemmas that were implied but not rigorously proven in \cite{BCLW}.  Section \ref{sec:ext} 
provides the proof of Theorems \ref{thm:ext} 
following similar steps as in \cite{BCLW} but filling in some details and providing more specificity about the probability with which results hold.
Section \ref{sec:intapprox} provides the detailed computations for the decay of the contour integrals outside a certain neighborhood of the critical point.  This is the most technical part of the paper and supplies key computations that were not addressed in \cite{BCLW}.  The lemma in this section is used in the proof in Section \ref{sec:ext}.  Finally, Section \ref{sec:2ol} provides a brief description of how this approach can be applied to the analysis of the overlap of two replicas.


\subsection*{Acknowledgements}
This work was supported in part by the NSF grants DMS-1701577 and DMS-1954790.  The author would like to thank Jinho Baik for all his helpful advice during the preparation of this paper.  The author would also like to thank Benjamin Landon for answering questions about his paper and Jeffrey Lagarias for providing support through his NSF grant.

\section{Summary of main theorem}\label{sec:results}

This paper focuses on the proof of Theorem \ref{thm:ext}, which provides the moment generating function for the overlap $\mgn$.  It is important to note that this overlap involve two types of randomness.  First, we have randomness from the choice of $M$ and $\efv$, which we refer to jointly as the ``disorder sample."  Second, we have randomness from the choice of spin variable.  For the results in this paper, we fix an arbitrarily disorder sample so that $\mgn$ is a random variable depending on a fixed disorder sample and random spin variable.  The moment generating function in Theorem \ref{thm:ext} provides the distribution of $\mgn$ as a function of the fixed disorder sample.  

This result is valid for an arbitrary disorder sample, subject to certain constraints that hold with high probability.  In particular, for any sufficiently small $\ep>0$, the event $\event$ (defined in Section \ref{sec:event}) provides a set of conditions on $M$ and $\efv$ that are sufficient for Theorem \ref{thm:ext} to hold.  Section \ref{sec:event} provides a detailed description of the event $\event$ along with a proof 
that 
\beq
\PP(\event)\geq1-N^{-\ep/10}\quad\text{for all sufficently small $\ep>0$ and all sufficiently large $N$.}
\eeq

\begin{theorem}\label{thm:ext} Given $T<1$ and $h=HN^{-1/2}$ for some some fixed $H\geq0$, we have the following asymptotic formula for the moment generating function of $\mgn$, the overlap with the external field.  This formula holds on the event $\event$ (which has probability at least $1-N^{-\ep/10}$) for any sufficiently small $\ep>0$ and $\xi=O(1)$.
\beq	
	\langle e^{ \xi \sqrt{N} \mgn} \rangle 
	=e^{ \efres \xi + \frac{ T \xi^2}{2} }
	\frac{\cosh \left( (\efres+T\xi) |n_1| \frac{\sqrt{1 - T}}{T} \right)}{\cosh \left(\efres |n_1| \frac{\sqrt{1 - T}}{T} \right) }
	\left(1+O(N^{-\frac{1}{21}+\frac\ep7})\right).
\eeq
\end{theorem}

Note that the leading term on the right-hand side is the product of two terms implying that it is the moment generating function of a sum of two independent random variables. 
The exponential term is the moment generating function of a Gaussian random variable. 
For the ratio of the cosh functions, we note that the moment generating function of a shifted Bernoulli random variable that takes values $1$ and $-1$ with probabilities $P$ and $1-P$ respectively is $P e^{t} + (1-P) e^{-t}$.  The ratio of cosh functions in Theorem \ref{thm:ext} is of this form with $t=\xi|n_1|\sqrt{1-T}$ and
\beq
P=\frac{e^{\frac{H}{T}|n_1|\sqrt{1-T}}}
{e^{\frac{H}{T}|n_1|\sqrt{1-T}}+e^{-\frac{H}{T}|n_1|\sqrt{1-T}}}.
\eeq
Hence, for any large $N$, we can conclude that, on the event $\event$, the scaled overlap $\sqrt{N} \mgn$ behaves in its leading order like the independent sum of a Gaussian random variable (with mean $H$ and variance $T$) and a shifted Bernoulli random variable (which takes values $|n_1|\sqrt{1-T}$ and $-|n_1|\sqrt{1-T}$ with probability $P$ and $1-P$ respectively for the value of $P$ stated above).  

We can use Theorem \ref{thm:ext} to obtain various information about the overlaps, including formulas for all moments of $\mgn$.  Of particular interest are the first moment (Gibbs expectation) and the variance.  Since $\mgn$ is of order $N^{-1/2}$ in the case of a microscopic external field, we examine the scaled overlap $\mgn\sqrt{N}$.  For the expectation, we get
\beq
\langle\mgn\sqrt{N}\rangle
=H+|n_1|\sqrt{1-T}\tanh\left(H|n_1|\frac{\sqrt{1-T}}{T}\right)
+O\left(N^{-\frac{1}{21}+\frac\ep3}\right)
\eeq
and for the variance, we get
\beq
\Var(\mgn\sqrt{N})
=T+n_1^2(1-T)\left(1-\tanh^2\left(H|n_1^2|\frac{\sqrt{1-T}}{T}\right)\right)+O\left(N^{-\frac{1}{21}+\frac\ep3}\right),
\eeq
where both of these formulas hold on the event $\event$.

\subsection{Application to magnetization and susceptibility}
\label{sec:scp}
One important application of Theorem \ref{thm:ext} is that it confirms the conjectures of \cite{BCLW} regarding magnetization and susceptibility.  Magnetization is defined to be $\langle\mgn\rangle$, the Gibbs average of the overlap with the external field.  Susceptibility is the magnetization per external field strength, given by
\beq
\mathcal{X}=\frac{\langle\mgn\rangle}{h}.
\eeq
It follows from Theorem \ref{thm:ext} that, on the event $\event$, when $T<1$ and $h=HN^{-1/2}$ for $H$ constant, the susceptibility is
\beq
\scp=1+\frac{|n_1|\sqrt{1-T}}{H}\tanh\left(H|n_1|\frac{\sqrt{1-T}}{T}\right)+O(N^{-\frac{1}{21}+\frac\ep3}).
\eeq
Of particular interest in the physics literature is the zero external field limit of the susceptibility.  Cugliandolo, Dean, and Yoshino \cite{cugliandolo2007nonlinear} discuss two ways to taking this limit, namely $\lim_{h\to0}\lim_{N\to\infty}\scp$ and $\lim_{N\to\infty}\lim_{h\to0}\scp$ (the first of these was also considered by \cite{kosterlitz1976spherical}).  Our results for the microscopic external field give a different way of calculating the second of these limits, namely with the computation
\beq\label{eq:limitingscp}
\lim_{H\to 0} \lim_{\substack{N\to \infty \\ h=HN^{-1/2}}} \scp
 = 1 +  \frac{ n_1^2  (1 - \tmp)}{\tmp}
\qquad \text{for $T<1$.}	
\eeq
This confirms the conjecture of \cite{BCLW}. It is also consistent with \cite{cugliandolo2007nonlinear}, which found that, after imposing the constraint $|n_1|=1$,
\beq
\lim_{N\to\infty}\lim_{h\to0}\scp=\frac1T
\qquad \text{for $T<1$}
\qquad
\cite{cugliandolo2007nonlinear}.
\eeq
By removing this constraint on $|n_1|$ and applying Theorem \ref{thm:ext}, we are able to show the dependence of the limiting susceptibility on the disorder sample in \eqref{eq:limitingscp}.  The paper \cite{BCLW} contains some further conjectures about the zero external field limit of differential susceptibility, which can also be verified using Theorem \ref{thm:ext}. 

\subsection{Comparison with the results of \cite{BCLW} and \cite{LS}}
Theorem \ref{thm:ext} is similar to Result 8.6 from \cite{BCLW} but we provide a more precise statement of the result as well as a rigorous proof.  In particular, we specify bounds for the order of the subleading term and the probability with which the result holds.  The proof of Theorem \ref{thm:ext} can be found in Section \ref{sec:ext}.  An additional lemma needed in the proof is included in section \ref{sec:intapprox}.

A similar approach can also be used to obtain a moment generating function for the overlap of two replicas.  In other words, we can adapt the methods from Sections \ref{sec:ext} and \ref{sec:intapprox} to prove a rigorous version of Result 10.6 from \cite{BCLW}.  This will be discussed further in Section \ref{sec:2ol}.  We do not provide the details of that proof because a comparable result was obtained via a different method in \cite{LS} (see Theorem 2.14).

\section{Preliminaries}\label{sec:prelim}
\subsection{Contour integral representations}\label{sec:integralreps}

Recall that the partition function $\pat_N$ is defined by the surface integral
\beq
	\pat_N = \int_{S_{N - 1}} e^{-\beta \ham(\spin)} \dd \omega_N(\sphv).
\eeq 
It was shown by Kosterlitz, Thouless and Jones \cite{kosterlitz1976spherical} that this surface integral can be rewritten as a contour intergral of the form
\beq
\label{eq:laplace_trans}
	\pat_N = C_N\int_{\cp - \ii \infty}^{\cp + \ii \infty} e^{\frac{N}{2}\G(z)}\dd z \quad \text{with} \quad C_N = \frac{\Gamma(N/2)}{2\pi \ii (N\beta/2)^{N/2 - 1}}
\eeq
where
\beq
\label{eq:def_G}
	\G(z) = \beta z - \frac1N\sum_{i = 1}^N \log(z - \eg_i) + \frac{\ef^2\beta}{N}\sum_{i = 1}^N \frac{n_i^2}{z - \eg_i}
\eeq
and the contour is a vertical line intersecting the real axis at any $\cp> \eg_1$.  This result can be extended in a straightforward way to obtain contour integral representations for generating functions of overlaps \cite{BCLW, LS}.  The results we use are given in the following lemma and a proof can be found in \cite{BCLW}.

\begin{lemma}[\cite{BCLW}] \label{lem:contour}
For real parameter $\eta$, the moment generation function of the overlaps is
\beq   \label{eq:contour_rep}
	\langle e^{\beta \eta \mgn}\rangle =  \frac{\int e^{\frac N2 \Gmgn (z)} \dd z}{\int e^{\frac N2 \G(z)} \dd z}
\eeq
where each contour is a vertical line passing to the right of all singularities and the function $\Gmgn(z)$
is defined as follows:
\beq\label{eq:G_defs}
\begin{aligned}
&\Gmgn(z) := \beta z - \frac1N\sum_{i = 1}^N \log(z - \eg_i) + \frac{(\ef + \frac{\eta}{N})^2\beta}{N} \sum_{i = 1}^N \frac{n_i^2}{z - \eg_i}\\
\end{aligned}
\eeq
Note that $\Gmgn(z)$ is $\G(z)$ with $h$ replaced by $h+\eta N^{-1}$.  
\end{lemma}

\subsection{Preliminaries from random matrix theory and probability}\label{sec:RMT}
In this section we present a few classical results from random matrix theory as well as some specific convergence results for certain functions of eigenvalues that we will use throughout the paper.
\subsubsection*{Semicircle Law}
For every bounded, continuous function $f(x)$, we have the following convergence of the empirical distribution of eigenvalues of $\sGOE$ \cite{Mehta}:   
	\beq
	\label{eq:lln}
		\frac{1}{N} \sum_{i = 1}^N f(\eg_i) \rightarrow \int f(x)\scl
		\quad \text{where} \quad 
		\scl = \frac{\sqrt{4 - x^2}}{2\pi}\mathbbm{1}_{x \in [-2,2]}\dd x
	\eeq
with probability $1$ as $N\to \infty$. 

\subsubsection*{Eigenvalue Rigidity}
For $i = 1,2,\cdots, N$, we let $\widehat{\eg}_i$ denote the classical location of the $i$th eigenvalue, defined by 
\beq	\label{eq:def_classical_eg}
		\int_{\widehat{\eg}_i}^2 \scl  = \frac i N.
\eeq
We set $\widehat{\eg}_0 = 2$.  The rigidity result  \cite{EYY, MR3098073} states that, for any $\e>0$ and $D>0$ and sufficiently large $N$,
\beq	\label{eq:rigidity}
	\PP\left(\bigcap_{i=1}^N \left\{|\eg_i- \widehat \eg_i| \leq  N^{-\frac23+\e}\left(\min\{i, N + 1 - i\}\right)^{-1/3}\right\}\right)\geq1-N^{-D}
\eeq

\subsubsection*{Airy Point Process}
Define the rescaled eigenvalues
\beq
\label{eq:scaledevii}
	a_i := N^{2/3}(\eg_i - 2). 
\eeq
As $N\to\infty$, the rescaled eigenvalues converge in distribution to the GOE Airy point process \cite{TracyWidom94, soshnikov1999universality}.  We denote this as  $\{\airy_i\}_{i = 1}^\infty$ satisfying 
\beq
	\{a_i\}\Rightarrow \{\airy_i\}.
\eeq
Heuristically, we expect that, for $1\ll i\ll N$,
\beq \label{eq:airypotamp}
		a_i\approx\airy_i \approx  -\left(\frac{3\pi i}{2} \right)^{2/3} 
\eeq
since the semicircle law is asymptotic to $\frac{\sqrt{2-x}}{\pi} \dd x$ as $x\to 2$. 
 The above approximation and the rigidity property suggest that, 	
\beq \label{eq:egresk}
		\egres_i \asymp - i^{2/3} \quad \text{as $i, N\to \infty$ satisfying $i\le N$}.
\eeq
For proofs throughout this paper, we need a more rigorous version of the approximation above, which we obtain in the following lemma.

\begin{lemma}\label{lem:LandonSosoe} (adapted from \cite{LandonSosoe})
There exist some integer $K$ and some $c>0$, which do not depend on $N$ such that, for all $k>K$, we have
\beq
\P\left(\bigcup_{N^{2/5}\geq j\geq k}\left\{a_1-a_j\geq cj^{2/3}\right\}\right)\geq1-\frac{2}{k^{1/2}}.
\eeq
\end{lemma}

\begin{proof}
In line (6.33) of \cite{LandonSosoe}, Landon and Sosoe obtain the result that there exists some $K_1$ (not depending on $N$) such that, for all $k>K_1$,
\beq\label{eq:landonsosoecorrected}
\P\left(\bigcap_{N^{2/5}\geq j\geq k}\left\{N^{2/3}(\lambda_j-2)\leq-\left(\frac{3\pi j}{2}\right)^{2/3}+\frac{1}{10}j^{2/3}\right\}\right)\geq1-\frac{1}{k^{1/2}}.
\eeq
(Note that the original statement of this inequality in the arxiv version of \cite{LandonSosoe} contains a typo, but the result above is what follows from the preceding lines of \cite{LandonSosoe} and we confirmed this with the authors.)
Next, we observe that there exists some $K'$ such that, for all $k>K'$, we have
\beq\label{eq:usingTWtaildecay}
\PP\left(N^{2/3}(2-\lambda_1)\leq\frac{1}{10}k^{2/3}\right)\geq1-\frac{1}{k^{1/2}}
\eeq
for $N$ sufficiently large.  This comes from the fact that the GOE Tracy-Widom distribution has sub-exponential tails.  Neither $K_1$ nor $K'$ depends on $N$, so we take $K$ to be the maximum of these two values and, combining \eqref{eq:landonsosoecorrected} and \eqref{eq:usingTWtaildecay}, we conclude the desired result.
\end{proof}

\subsubsection*{Special sums}
There are a few sums that will be particularly important throughout this paper.  Below we present some convergence results for these sums, which depend upon the random matrix properties described above.

In particular,  for $m = 1, 2, \cdots $, we consider sums of the form  
\beq \label{eq:rationalfunk}
	\frac1N \sum_{i = 2}^N \frac{1}{(\eg_1 - \eg_i)^m} .
\eeq
We need an asymptotic formula for $m=1$ as $N\to\infty$.  This was obtained recently in \cite{LandonSosoe}. Landon and Sosoe proved that 
	\beq
	\label{eq:def_crvbp}
		\Xi_N := N^{1/3} \left(\frac1N \sum_{i = 2}^N \frac{1}{\eg_1 - \eg_i} - 1\right) \Rightarrow \Xi
	\eeq 
for a random variable $\Xi$ as $N\to \infty$. 
The limiting random variable $\Xi$ can be expressed in terms of the GOE Airy kernel point process as 
	\beq
	\label{eq:def_crvbplim}
		\Xi = \lim_{n \rightarrow \infty} \left(\sum_{i = 2}^n \frac{1}{\airy_1 - \airy_i} - \frac1{\pi} \int_0^{\left(\frac{3\pi n}{2}\right)^{2/3}} \frac{\dd x}{\sqrt{x}}\right)
	\eeq
where the limit exists almost surely.  

We also need another version of the result \eqref{eq:def_crvbp} where the constant numerators are replaced $n_i^2$:
\beq
\label{eq:weightedsum1}
	N^{1/3}\left(\frac1N \sum_{i = 2}^N \frac{n_i^2}{\eg_1 - \eg_i} - 1\right) \Rightarrow  \lim_{n \rightarrow \infty} \left(\sum_{i = 2}^n \frac{\nu_i^2}{\airy_1 - \airy_i} - \frac1{\pi} \int_0^{\left(\frac{3\pi n}{2}\right)^{2/3}} \frac{\dd x}{\sqrt{x}}\right)
\eeq
where $\nu_i$ are i.i.d standard Gaussians, independent of the GOE Airy point process $\alpha_i$.  
This follows from \eqref{eq:def_crvbp} and the fact that 
\beq\label{eq:convdiffsum}
	\frac1{N^{2/3}} \sum_{i = 2}^N \frac{n_i^2-1}{\eg_1 - \eg_i} \Rightarrow \sum_{i = 2}^\infty \frac{\nu_i^2-1}{\airy_1 - \airy_i}
\eeq
which is a convergent series due to Kolmogorov's three series theorem and Lemma \ref{lem:LandonSosoe}. 


Next, we have two lemmas concerning the convergence of special sums.

\begin{lemma}\label{lem:specialsumLS}
For any $\e>0$,
\beq
	\label{eq:res_edge}
		\frac1N \sum_{i = 2}^N \frac{1}{\eg_1 - \eg_i} = 1 + O(N^{-\frac13+\e})
		\quad\text{and}\quad
		\frac1N \sum_{i = 2}^N \frac{n_i^2}{\eg_1 - \eg_i} = 1 + O(N^{-\frac13+\e})
	\eeq
with probability at least $1-N^{-\e/2}$. (This lemma is adapted from a similar result in \cite{LS}).
\end{lemma}

\begin{proof}
Define an event
\beq
E_\e:=\left\{\lambda_1-\lambda_2\geq N^{-\frac23(1+\e)}\right\}\cap
\left\{ \bigcap_{i=1}^N \left\{|\eg_i- \widehat \eg_i| \leq  N^{-\frac23+\e}\left(\min\{i, N + 1 - i\}\right)^{-1/3}\right\} \right\}.
\eeq
The first equation in \eqref{eq:res_edge} holds on this event, which we can see by writing
\beq\begin{split}
\frac1N \sum_{i = 2}^N \frac{1}{\eg_1 - \eg_i}
&=\frac1N \sum_{i = 2}^{N^{\e/3}} \frac{1}{\eg_1 - \eg_i} + 
\frac1N \sum_{i = N^{\e/3}+1}^N \frac{1}{\eg_1 - \eg_i}\\
&=O(N^{-\frac13+\e})+\left(1+O(N^{-\frac13+\e})\right)
\end{split}\eeq 
where, for the first sum, we use $\lambda_1-\lambda_i\geq N^{-\frac23(1+\e)}$ and, for the second sum, we use eigenvalue rigidity and the semicircle law.  The second equation in \eqref{eq:res_edge} also holds on $E_\e$ using the same reasoning along with the fact that the sum in \eqref{eq:convdiffsum} is convergent.  It remains only to show that 
\beq
\PP(E_\e)\geq1-N^{-\e/2}.
\eeq
From Lemma 3.4 from \cite{LandonSosoe}, we have
\beq
\PP(\lambda_1-\lambda_2\geq N^{-\frac23(1+\e)})\geq1-N^{-\frac23 \e+\e'}
\eeq
for any $\e'>0$.  This, along with \eqref{eq:rigidity}, implies the lemma.
\end{proof}
We also consider a similar class of sums with a larger exponent in the denominator and get the following lemma.

\begin{lemma}\label{lem:specialsum} For any $\e>0$
\beq
	\sum_{i=2}^N\frac{1}{(a_1-a_i)^m}=O(N^\e) \quad \text{and} \quad 
	\sum_{i=2}^N\frac{n_i^2}{(a_1-a_i)^m}=O(N^\e)  , \qquad\text{$m\geq2$,}
\eeq
with probability at least $1-N^{-\frac{\e}{3m}}$.
 \end{lemma}
\begin{proof}
To prove the first of these inequalities we consider the event
\beq
\cF_\e=\left\{a_1-a_2>N^{-\frac{\e}{2m}}\right\}\cap
\left\{\bigcap_{i=1}^N \left\{|\eg_i- \widehat \eg_i| \leq  N^{-\frac23+\e}\left(\min\{i, N + 1 - i\}\right)^{-1/3}\right\}\right\}
\eeq
The event $\left\{a_1-a_2>N^{-\frac{\e}{2m}}\right\}$ occurs with probability at least $1-N^{-\frac{\e}{2m}+\\e'}$ for any $\e'>0$ (see \cite{LandonSosoe} Lemma 3.4).  Using this fact along with the eigenvalue rigidity result \eqref{eq:rigidity}, we can conclude that the event $\cF_\e$ occurs with probability at least $1-N^{-\frac{\e}{3m}}$.  Now we show that the first inequality in \eqref{lem:specialsum} holds on the event $\cF_\e$.  In particular, on that event, we have
\beq
\begin{split}
\sum_{i=2}^N\frac{1}{(a_1-a_i)^m}
&=\sum_{i=2}^{N^{\e/2}}\frac{1}{(a_1-a_i)^m}+\sum_{i=N^{\e/2}}^N\frac{1}{(a_1-a_i)^m}\\
&<N^{\e/2}\cdot \frac{1}{(N^{-\frac{\e}{2m}})^m}+2\sum_{i=N^{\e/2}}^N\frac{1}{(-a_i)^m}\\
&\leq N^\e+2\sum_{i=N^{\e/2}}^N\frac{1}{(-\hat{a}_i)^m}\left(1+\frac{|a_i^m-\hat{a}_i^m|}{(-a_i)^m}\right)\\
&< N^\e+4\sum_{i=N^{\e/2}}^N\frac{1}{(-\hat{a}_i)^m}
\end{split}
\eeq
The summation in the last line is well approximated by the integral
\beq
N^{-\frac{2m}{3}+1}\int_{-2}^{\lambda_{N^{\e/2}}}\frac{1}{(2-x)^m}d\sigma_{SC}(x) <4N^{-\frac{2m}{3}+1}\int_{-2}^{\lambda_{N^{\e/2}}}\frac{1}{(2-x)^{m-\frac12}}dx
\eeq
Using the approximation $2-\lambda_{N^{\e/2}}\approx cN^{-\frac23+\frac{\e}{3}}$ from \eqref{eq:airypotamp}, we see that the right hand side of the inequality above is of order $N^{-m\e/3}$.  Thus $\sum_{i=2}^N\frac{1}{(a_1-a_i)^m}=O(N^\e)$ on the event $\cF_\e$.  Because $n_i$ are standard Gaussians, the sum $\sum_{i=2}^N\frac{n_i^2}{(a_1-a_i)^m}$ has the same order with comparable probability.
\end{proof}

\subsubsection*{Chi-squared distributions}
One quantity that we make use of throughout this paper is $n_1=\mathbf{u}_1^T\mathbf{g}$.  We note that $n_1$ has a standard normal distribution which means that $n_1^2$ has a chi-squared distribution with one degree of freedom.  We prove many results that hold on the event where $n_1^2$ is roughly of order 1.  More specifically, we have the following lemma
\begin{lemma}\label{lem:n1bound} For any sufficiently small $\e>0$, 
\beq
\PP\left(N^{-\e}<n_1^2<\e\log N\right)\geq 1-N^{-\e/2}
\eeq
\end{lemma}
The proof of this lemma is straightforward from the probability density function for chi-squared random variables.  We note for the purpose of future results that $n_1$ is independent of the eigenvalues of $M$.

\subsection{Defining the event on which our result holds}\label{sec:event}
For any $\ep>0$, we define an event $\event $ as follows:
\begin{multline}\label{eq:eventdef}
\event :=\left\{N^{-\ep}<n_1^2<\ep\log N\right\}\cap
\left\{\frac1N\sum_{i=2}^N\frac{1}{\lambda_1-\lambda_i}=1+O(N^{-\frac13+\ep})
\text{ and }\frac1N\sum_{i=2}^N\frac{n_i^2}{\lambda_1-\lambda_i}=1+O(N^{-\frac13+\ep})\right\}\\
\cap\left\{\sum_{i=2}^N\frac{1}{(a_1-a_i)^m}\leq N^{\ep}
\text{ and }\sum_{i=2}^N\frac{n_i^2}{(a_1-a_i)^m}\leq N^{\ep}\text{ for }m=2,3\right\}
\end{multline}

\begin{lemma}\label{lem:eventprob}
For $\ep>0$ sufficiently small and $N$ sufficiently large,
\beq
\PP(\event )\geq1-N^{-\ep/10}
\eeq
\end{lemma}
\begin{proof}
The event $\event$ as defined above is the intersection of three events, each with probability close to 1.  For sufficiently large $N$, we know from Lemma \ref{lem:n1bound} that the first event in the intersection has probability at least $1-N^{-\ep/2}$ and, from Lemma \ref{lem:specialsumLS}, the second event in the intersection has probability at least $1-N^{-\ep/2}$.  The third event in the intersection is actually composed of two events, the one for $m=2$ and the one for $m=3$.  By Lemma \ref{lem:specialsum}, these hold with probability $1-N^{-\ep/6}$ and $1-N^{-\ep/9}$ respectively.  Putting these together, we see that, even if the complements of all of these events are disjoint, we have 
$\PP(\event )\geq1-N^{-\ep/10}$
for any sufficiently small $\ep>0$ and sufficiently large $N$. 
\end{proof}

 Throughout the rest of this paper, we will prove various results assuming that we are on the event $\event $.  We can then conclude that those results hold with probability at least $1-N^{-\ep/10}$.
\section{Proof of Theorem \ref{thm:ext}}\label{sec:ext}

In the proof of Theorem \ref{thm:ext}, we make use of Lemma \ref{lem:contour}, which can be restated as follows:
\beq \label{eq:ovwefdn}
	\langle e^{\beta \xi \sqrt{N} \mgn}\rangle = e^{\frac N2(\Gmgn(\cpmgn) - \G(\cp))} \frac{\int_{\cpmgn - \ii\infty}^{\cpmgn + \ii\infty} e^{\frac N2 (\Gmgn (z) - \Gmgn(\cpmgn))} \dd z}{\int_{\cp - \ii \infty}^{\cp + \ii \infty} e^{\frac N2 (\G(z) - \G(\cp))} \dd z}
\eeq
where 
\beq 
	\Gmgn(z) = \beta z - \frac1N\sum_{i = 1}^N \log(z - \eg_i) + \frac{(\ef + \frac{\xi}{\sqrt{N}})^2\beta}{N} \sum_{i = 1}^N \frac{n_i^2}{z - \eg_i}
\eeq
and we use $\cp$ and $\cpmgn$ to denote critical points $\G(z)$ and $\Gmgn(z)$ respectively, which satisfy $\cp>\lambda_1$ and $\cpmgn>\lambda_1$.  In the next two lemmas, we show that these critical points are unique and we compute upper and lower bounds for them.  After accomplishing this, we turn to the more delicate task of computing the integrals in the formula for the generating function of $\mgn$.  This is more difficult for $h\sim N^{-1/2}$ than in the other scaling regimes because the critical point is very close to a branch point. Since a straightforward application of Taylor approximation and steepest descent analysis does not work in this case, we directly compute the integral in a neighborhood of the critical point and then show that the tails of the integral are of smaller order.

\subsection{Critical point analysis}
We begin by computing the critical point $\cp$ of $\G(z)$.  In \cite{BCLW}, the authors use the ansatz that $\cp = \eg_1 + \so N^{-1}$ with $N^{-\e}<\so<N^{\e}$ for any $\e>0$ and sufficiently large $N$ on some event whose probability tends to 1 as $N\to\infty$.  Without making any assumption about the order of $\so$, we set
\beq \label{eq:sohca}
	\cp = \eg_1 + \so N^{-1}
\eeq 
and then prove that the order of $\so$ indeed satisfies the ansatz of \cite{BCLW} (in fact we prove something more precise).  In particular, we can define $\so$ via the the formula for $\G'(z)$ and prove the following lemma.

\begin{lemma}\label{lem:ext1/2cp} There exists a unique $\so>0$ satisfying the equation
\beq\label{eq:ext1/2sdef}
	\G'(\eg_1+\so N^{-1}) = \beta  - \frac1{N} \sum_{i=1}^N \frac{1}{\eg_1+\so N^{-1}-\eg_i} - \frac{H^2\beta}{N^2} \sum_{i=1}^N \frac{n_i^2}{(\eg_1+\so N^{-1}-\eg_i)^2} =0 
\eeq
 and, for any sufficiently small $\ep>0$ and sufficiently large $N$, we have $T<\so<\ep\log N$ on the event $\event$, which occurs with probability at least $1-N^{-\ep/10}$.
\end{lemma}
\begin{proof} On the event $\event$, the last sum in equation \eqref{eq:ext1/2sdef} is $O(N^{-\frac23+\ep})$ for any sufficiently small $\ep>0$. From this, we get
\beq
\beta-\frac{1}{N}\sum_{i=2}^N\frac{1}{\lambda_1-\lambda_i}-\frac{1}{\so }-\frac{H^2\beta n_1^2}{\so ^2}+O(N^{-\frac23+\ep})\;<\;0\;<\;\beta-\frac{1}{\so }-\frac{H^2\beta n_1^2}{\so ^2}
\eeq
on $\event$.  Further applying definition of $\event$ to the sum on the left hand side and rearranging terms, we get
\beq
\beta-1+O(N^{-\frac13+\ep})\;<\;\frac{1}{\so }+\frac{H^2\beta n_1^2}{\so ^2}\;<\;\beta.
\eeq
Hence, on $\event$, the expression $\frac{1}{\so }+\frac{H^2\beta n_1^2}{\so ^2}$ is bounded above and below by order 1 quantities.  The upper bound ensures that $\so>\frac1\beta=T$ (note this is not a sharp bound).  The lower bound on $\frac{1}{\so }+\frac{H^2\beta n_1^2}{\so ^2}$ ensures that $\so =O(\ep\log N)$ provided that $|n_i|=O(\ep\log N)$.  Since $|n_i|<(\ep\log N)^{1/2}$ for sufficiently large $N$ on $\event$, we can definitely ensure that $|n_i|<C\ep\log N$ for any constant $C$ and sufficiently large $N$.
\end{proof}
Having proved the lemma, we apply the bounds on the order of $p$ to equation \eqref{eq:ext1/2sdef} and conclude that $\so$ satisfies
\beq
\label{eq:shN1/2}
	\beta - 1 - \frac{1}{\so } - \frac{\efres^2 \beta n_1^2}{\so ^2} + O(N^{-\frac13+\ep}) =0
\eeq
with probability $1-N^{-\ep/10}$.
We note that, when $h=HN^{-1/2}$, the equation for $\Gmgn$ is same as the one for $\G$ with $H$ replaced by $H+\xi$. 
Thus  $\cpmgn = \eg_1 + \so_\mgn N^{-1}$ where $\so_\mgn>0$ solves the equation 
\beq
\label{eq:smgnhN1/2}
	\beta - 1 - \frac{1}{\so_\mgn} - \frac{(\efres + \xi)^2 \beta n_1^2}{\so_\mgn^2} + O(N^{-\frac13+\ep}) = 0,
\eeq
and the lemma below follows by the same reasoning as in the lemma above.

\begin{lemma}\label{lem:ext1/2cpmgn} There exists a unique $\so_\mgn>0$ satisfying the equation
\beq\label{eq:ext1/2smgndef}
	\Gmgn'(\eg_1+\so_\mgn N^{-1}) = \beta  - \frac1{N} \sum_{i=1}^N \frac{1}{\eg_1+\so_\mgn N^{-1}-\eg_i} - \frac{(H+\xi)^2\beta}{N^2} \sum_{i=1}^N \frac{n_i^2}{(\eg_1+\so_\mgn N^{-1}-\eg_i)^2} =0 
\eeq
 and, for any sufficiently small $\ep>0$ and sufficiently large $N$, we have $T<\so_\mgn<\ep\log N$ on the event $\event$, which occurs with probability at least $1-N^{-\ep/10}$.
\end{lemma}

\subsection{Contour integral computation}

We now consider the ratio of the integrals in the formula \eqref{eq:contour_rep}. 
For the integral in the numerator, we have the following lemma.
\begin{lemma}\label{lem:ext1/2integral} For fixed $H>0$ with $h=HN^{-1/2}$ and $T<1$
\beq	
\label{eq:int_const}
	\int e^{\frac{N}2 (\Gmgn(z)- \Gmgn(\cpmgn))} \dd z 
	=
	\frac{2\ii \sqrt{2 \pi \so_\mgn}  e^{- (\beta-1)\so_\mgn +\frac12 } }{N\sqrt{\beta-1}} 
	\cosh \left( (\efres+\xi) |n_1| \sqrt{\beta (\beta -1)} \right)
	\left(1+O(N^{-\frac{1}{21}+\frac\ep7})\right)
\eeq
on the event $\event$, which occurs with probability at least $1-N^{-\ep/10}$ for any sufficiently small $\ep>0$.
\end{lemma}
\begin{proof}
To compute this integral, we need a formula for $N(\Gmgn(z)-\Gmgn(\cpmgn))$
in terms of $u$ where $z=\cpmgn+uN^{-1}$.  We will begin by focusing on the central portion of the integral and then we will handle the tails separately.  When we are on the event $\event$ and $|u|=o(N^{\frac13-\ep})$, we get the following computation: 
\beq\begin{split}
N&(\Gmgn(z)-\Gmgn(\cpmgn))
=N( \Gmgn(z)-  \Gmgn(\cpmgn)- \Gmgn'(\cpmgn)uN^{-1})  \\
& = -  \sum_{i=1}^N \left[ \log \left( 1+ \frac{uN^{-1}}{\cpmgn-\eg_i} \right) -  \frac{uN^{-1}}{\cpmgn-\eg_i}  \right] 
+\frac{(H+\xi)^2\beta}{N}
\sum_{i=1}^N  \frac{n_i^2 u^2N^{-2}}{(\cpmgn+uN^{-1}-\eg_i)(\cpmgn-\eg_i)^2}\\
&=-\log\left(1+\frac{u}{\so_\mathfrak{M}}\right)+\frac{u}{\so_\mathfrak{M}}
+O\left(\sum_{j=2}^N\frac{|u|^2N^{-2}}{(\cpmgn-\lambda_j)^2}\right)
+\frac{(H+\xi)^2\beta n_1^2u^2}{(\so_\mathfrak{M}+u)\so_\mathfrak{M}^2}
+O\left(\sum_{j=2}^N\frac{|u|^2N^{-3}}{(\cpmgn-\lambda_j)^3}\right)\\
\end{split}\eeq
Using properties of the event $\event$, the quantity in the last line above can be simplified as follows:
\beq\begin{split}
&=-\log\left(1+\frac{u}{\so_\mathfrak{M}}\right)+\frac{u}{\so_\mathfrak{M}}
+\frac{(H+\xi)^2\beta n_1^2u^2}{(\so_\mathfrak{M}+u)\so_\mathfrak{M}^2}+O(|u|^2N^{-\frac23+\ep})\\
&=-\log\left(1+\frac{u}{\so_\mathfrak{M}}\right)+(\beta-1+O(N^{-\frac13+\ep}))(u-\so_\mathfrak{M})+1
+\frac{(H+\xi)^2\beta n_1^2}{(\so_\mathfrak{M}+u)}+O(|u|^2N^{-\frac23+\ep})\\
&=-\log\left(1+\frac{u}{\so_\mathfrak{M}}\right)+(\beta-1)(u-\so_\mathfrak{M})+1
+\frac{(H+\xi)^2\beta n_1^2}{(\so_\mathfrak{M}+u)}+O\left((|u|+1)N^{-\frac13+\ep}\right).
\end{split}\eeq
Now, set $f(u)=\frac{1}{2}\left(-\log\left(1+\frac{u}{\so_\mathfrak{M}}\right)+(\beta-1)(u-\so_\mathfrak{M})+1
+\frac{(H+\xi)^2\beta n_1^2}{(\so_\mathfrak{M}+u)}\right)$ and let $0<\e<\frac16-\frac\ep2$.  Then we see that, on the event $\event$,
\beq\begin{split}
\int_{\cpmgn-\ii\infty}^{\cpmgn+\ii\infty}\exp&\left[\frac{N}{2}(\Gmgn(z)-G_m(\cpmgn))\right]\dd z\\
=&\frac{1}{N}\left(\int_{-\ii N^\e}^{\ii N^\e}\exp\left(f(u)+O\left((|u|+1)N^{-\frac13+\ep}\right)\right)\dd u
+\text{ integrals of tails}\right).
\end{split}\eeq
For the purposes of computing this, it helps to deform the contour by shifting it leftward so that, instead of the vertical contour from $\cpmgn-\ii\infty$ to $\cpmgn+\ii\infty$, we consider the contour from $\lambda_1-\ii\infty$ to $\lambda_1+\ii\infty$ which is a straight vertical line except near $\lambda_1$ where it passes to the right of the branch point.  The integral on this contour will be the same as on the original contour and we get
\beq\begin{split}
\int_{\lambda_{1+}-\ii\infty}^{\lambda_{1+}+\ii\infty}\exp&\left[\frac{N}{2}(\Gmgn(z)-\Gmgn(\cpmgn))\right]\dd z\\
=&\frac{1}{N}\left(\int_{0_+-\ii N^\e}^{0_++\ii N^\e}\exp\left(f(u-\so_\mathfrak{M})+O\left((|u|+1)N^{-1/3}\right)\right)\dd u+\text{ integrals of tails}\right)\\
=&\frac{1}{N}\left(\int_{0_+-\ii N^\e}^{0_++\ii N^\e}\exp(f(u-\so_\mathfrak{M}))\left(1+O\left((|u|+1)N^{-\frac13+\ep}\right)\right)\dd u+\text{ integrals of tails}\right).
\end{split}\eeq
Next, we compute the integral on the portion of the contour from $-\ii N^\e$ to $\ii N^\e$.  Call this portion of the contour $C$.  We define $C$ more specifically to be composed of  three pieces:
\begin{itemize}
\item $C_1$ is the straight line from $\lambda_1-\ii N^\e$ to $\lambda_1-\ii \so_\mathfrak{M}$.
\item $C_2$ is the semicircle given by $\lambda_1+\so_\mathfrak{M}e^{\ii\theta}$ with $\theta\in[-\frac{\pi}{2},\frac{\pi}{2}]$.
\item $C_3$ is the straight line from $\lambda_1+\ii \so_\mathfrak{M}$ to $\lambda_1+\ii N^\e$.
\end{itemize}
We show that $\exp(f(u-\so_\mathfrak{M}))$ is bounded on $C_1,C_2,C_3$ by bounding the real part of $f(u-\so_\mathfrak{M})$.  On $C_1$ and $C_3$, we have
\beq
\Re(f(u-\so_\mathfrak{M}))=-\frac{1}{2}\log\left(\frac{|u|}{\so_\mathfrak{M}}\right)-2\so_\mathfrak{M}(\beta-1)+1<1.
\eeq
On $C_2$, we have
\beq\begin{split}
\Re(f(\so_\mathfrak{M}e^{\ii\theta}-\so_\mathfrak{M}))
=&-\frac{1}{2}\log(|e^{\ii\theta}|)+(\beta-1)\cdot\Re(\so_\mathfrak{M}e^{\ii\theta}-2\so_\mathfrak{M})+1+\Re\left(\frac{(H+\xi)^2\beta n_1^2}{\so_\mathfrak{M}e^{\ii\theta}}\right)\\
<&1+\frac{(H+\xi)^2\beta n_1^2}{\so_\mathfrak{M}}.
\end{split}\eeq
Since the real part of $f(u-\so_\mathfrak{M})$ is bounded, the magnitude of $\exp(f(u-\so_\mathfrak{M}))$ is also bounded by some constant (call it $c$) so we have
\beq\begin{split}
\int_{C_1}\exp&(f(u-\so_\mathfrak{M}))\left(1+O\left((|u|+1)N^{-\frac13+\ep}\right)\right)\dd u\\
=&\int_{-\ii N^\e}^{-\ii \so_\mathfrak{M}}\exp(f(u-\so_\mathfrak{M}))\left(1+O\left((|u|+1)N^{-\frac13+\ep}\right)\right)\dd u\\
=&\int_{-\ii N^\e}^{-\ii \so_\mathfrak{M}}\exp(f(u-\so_\mathfrak{M}))\dd u+O(c\cdot2N^{2\e-\frac13+\ep})\\
=&\int_{C_1}\exp(f(u-\so_\mathfrak{M}))\dd u+O(N^{2\e-\frac13+\ep}).
\end{split}\eeq
Similarly, we have
\beq
\int_{C_3}\exp(f(u-\so_\mathfrak{M}))\left(1+O\left((|u|+1)N^{-\frac13+\ep}\right)\right)\dd u=\int_{C_3}\exp(f(u-\so_\mathfrak{M}))\dd u+O(N^{2\e-\frac13+\ep}).
\eeq
Finally, for $C_2$, we get
\beq\begin{split}
\int_{C_2}\exp&(f(u-\so_\mathfrak{M}))\left(1+O\left((|u|+1)N^{-\frac13+\ep}\right)\right)\dd u\\
=&\int_{-\pi/2}^{\pi/2}\exp(f(\so_\mathfrak{M}e^{\ii\theta}-\so_\mathfrak{M}))\left(1+O\left(N^{-\frac13+\ep}\right)\right)\so_\mathfrak{M}\ii e^{\ii\theta}\dd\theta\\
=&\int_{-\pi/2}^{\pi/2}\exp(f(\so_\mathfrak{M}e^{\ii\theta}-\so_\mathfrak{M}))\so_\mathfrak{M}\ii e^{\ii\theta}\dd\theta+O\left(\so_\mathfrak{M}c\pi N^{-\frac13+\ep}\right)\\
=&\int_{C_2}\exp(f(u-\so_\mathfrak{M}))\dd u+O(N^{-\frac13+\ep}).
\end{split}\eeq
Thus, we conclude that, on the event $\event$, for any $0<\e<\frac16-\frac\ep2$,
\beq
\int_{0_+-\ii N^\e}^{0_++\ii N^\e}\exp(f(u-\so_\mathfrak{M}))\left(1+O\left((|u|+1)N^{-\frac13+\ep}\right)\right)\dd u=\int_{0_+-\ii N^\e}^{0_++\ii N^\e}\exp(f(u-\so_\mathfrak{M}))\dd u+O(N^{2\e-\frac13+\ep}).
\eeq
We use lemma \ref{lem:intapproxspecialcase} to show that the integral of the tails has order $O(N^{-\e/3})$.  This has the same order as $O(N^{2\e-\frac13+\ep})$ when $\e=\frac{1}{7}(1-3\ep)$, which is positive for any $0<\ep<\frac13$.  Since we are free to choose any $0<\e<\frac16$, we set $\e=\frac{1}{7}(1-3\ep)$ and get
\beq\label{eq:ext1/2contourcalc}
\begin{aligned}
	\int& e^{\frac{N}{2}(\Gmgn(z) - \Gmgn(\cpmgn))} \dd z 
	=\frac 1N \left(\int_{0_+-\ii N^\e}^{0_++\ii N^\e}\exp(f(u-\so_\mathfrak{M}))\dd u+O\left(N^{-\frac{1}{21}+\frac\ep7}\right) \right)\\
	&=  \frac1{N} \left(\int_{0_+-\ii N^\e}^{0_++\ii N^\e} \sqrt{\frac{\so_\mgn}{\so_\mgn + u}} e^{\frac{(\beta - 1)(u - \so_\mgn)}{2} + \frac12 + \frac{(\efres + \xi)^2 \beta n_1^2}{2(\so_\mgn + u)}} \dd u +O\left(N^{-\frac{1}{21}+\frac\ep7}\right) \right) \\
	&= \frac{\so_\mgn ^{1/2}e^{-(\beta - 1)\so_\mgn + \frac12 }}{N} \left(\int_{0_+-\ii N^\e}^{0_++\ii N^\e} \frac{1}{\sqrt{\so_\mgn + u}} e^{\frac{(\beta - 1)(\so_\mgn + u)}{2} + \frac{(\efres + \xi)^2 \beta n_1^2}{2(\so_\mgn + u)}} \dd u+ O\left(N^{-\frac{1}{21}+\frac\ep7}\right) \right).
\end{aligned}
\eeq
The integral $ \int_{0_++\ii\R} \frac{1}{\sqrt{\so_\mgn + u}} \exp\left(\frac{(\beta - 1)(\so_\mgn + u)}{2} + \frac{(\efres + \xi)^2 \beta n_1^2}{2(\so_\mgn + u)}\right) \dd u$ can be evaluated using the contour integral formula for the modified Bessel function (see e.g. \cite{Handbook}): 
\beq	\label{eq:ntif}
	\int_{0_++\ii\R} \frac1{\sqrt{w}} e^{aw+\frac{b}{w}} \dd w = 2 \pi \ii \left( \frac{b}{a} \right)^{1/4} I_{-\frac12}(2\sqrt{ab}) = \frac{2\ii\sqrt{\pi}}{\sqrt{a}} \cosh(2\sqrt{ab}).
\eeq
Since this integral converges, the integral in the last line of equation \eqref{eq:ext1/2contourcalc} must converge to the same value.  Furthermore, the tails of the integral in \ref{eq:ntif} beyond order $N^\e$ only contribute $O(N^{-\e/2})$ to the value of the intergral.  This is less than $O(N^{-\frac{1}{21}+\frac\ep7})$ since we set $\e=\frac17 (1-3\ep)$.  Hence, we conclude that, on the event $\event$,
\beq	
\label{eq:int_const}
	\int e^{\frac{N}2 (\Gmgn(z)- \Gmgn(\cpmgn))} \dd z 
	=\frac{2\ii \sqrt{2 \pi \so_\mgn}  e^{- (\beta-1)\so_\mgn +\frac12 } }{N\sqrt{\beta-1}} 
	\cosh \left( (\efres+\xi) |n_1| \sqrt{\beta (\beta -1)} \right)\left(1+O(N^{-\frac{1}{21}+\frac\ep7})\right).
\eeq
\end{proof}

We now return to the task of computing the moment generating function of $\mgn$ using the formula in line \eqref{eq:ovwefdn}.  The integral in the denominator of that formula can be viewed as a special case of the numerator in which $\xi=0$ and $\so_\mgn$ is replaced with $p$.
Therefore, on the event $\event$,
\beq
	\frac{\int e^{\frac{N}{2}(\Gmgn(z) - \Gmgn(\cpmgn))} \dd z}{\int e^{\frac{N}{2}(\G(z) - \G(\cp))} \dd z} = \sqrt{\frac{\so_\mgn}{\so }} e^{-(\beta - 1)(\so_\mgn - \so )} \frac{\cosh \left( (\efres+\xi) |n_1| \sqrt{\beta (\beta -1)} \right)}{\cosh \left(\efres |n_1| \sqrt{\beta (\beta -1)} \right) }\left(1+O(N^{-\frac{1}{21}+\frac\ep7})\right).
\eeq

To compute the moment generating function of $\mgn$ from the formula \eqref{eq:ovwefdn} it remains only to evaluate the factor $e^{\frac{N}{2}(\Gmgn(\cpmgn)-\G(\cp))}$.  This is computed in \cite{BCLW} and the authors find the following (Note that in their computation is less precise about the order of the Big-$O$ term. However, it can easily be made rigorous by repeating their steps using the assumptions that hold on the event $\event$ and carefully tracking the order of each term.  This yields the result below):
\beq
\begin{aligned}
	N(\Gmgn(\cpmgn)-\G(\cp))
	= & - \log(\frac{\so_\mgn}{\so}) + 2(\beta - 1) (\so_\mgn - \so) + (2\efres \xi + \xi^2) \beta + O(N^{-\frac13+\ep}).
\end{aligned}
\eeq
Thus we can conclude that 
\beq
	\langle e^{\beta \xi \sqrt{N} \mgn} \rangle 
	=e^{  \frac{ (2H\xi+ \xi^2) \beta}{2} }
	\frac{\cosh \left( (\efres+\xi) |n_1| \sqrt{\beta (\beta -1)} \right)}{\cosh \left(\efres |n_1| \sqrt{\beta (\beta -1)} \right) }
	\left(1+O(N^{-\frac{1}{21}+\frac\ep7})\right).
\eeq
Replacing $\beta \xi$ by $\xi$ and using $T=1/\beta$, we obtain
\beq	
	\langle e^{ \xi \sqrt{N} \mgn} \rangle 
	=e^{ \efres \xi + \frac{ T \xi^2}{2} }
	\frac{\cosh \left( (\efres+T\xi) |n_1| \frac{\sqrt{1 - T}}{T} \right)}{\cosh \left(\efres |n_1| \frac{\sqrt{1 - T}}{T} \right) }
	\left(1+O(N^{-\frac{1}{21}+\frac\ep7})\right).
\eeq
This gives us the result stated in Theorem \ref{thm:ext}.


\section{Integral Approximation Lemmas}\label{sec:intapprox}

The proof of Theorems \ref{thm:ext} in the preceding section required us to compute a contour integral.  In that computation, we relied on the fact that the dominant contribution to the integral comes from a neighborhood of the critical point.  
In this section, we prove that fact by providing an upper bound for the value of the integral outside of a neighborhood of the critical point and showing that the upper bound shrinks to zero as $N\to\infty$.  This is the most technical part of the contour integral computations.

\begin{lemma}\label{lem:intapproxspecialcase} Tail approximation for overlap with external field when $h=H^{-1/2}$ and $T<1$:  For any $\e>0$,
\beq
\int_{\ii N^\e}^{\ii\infty}\exp\left[\frac N2(\Gmgn(\lambda_1+uN^{-1})-\Gmgn(\cpmgn))\right]\dd u=O(N^{-\e/3})\quad\text{as }N\to\infty
\eeq
on the event $\event$.
\end{lemma}
\begin{proof} First, observe that we are using a vertical contour with real part equal to $\l_1$ as opposed to the original contour, which had real part equal to $\g$.  This is due to a contour deformation that we did when computing the integral on the central portion of the contour. To show that the integrals of the tails tend to zero, we deform the contour yet again.  Instead of the vertical line contour given by $\lambda_1+i(N^\e+t)N^{-1}$ with $t\in[0,\infty)$, we consider the contour $C_4$ given by $\lambda_1-f(t)N^{-1}+i(N^\e+t)N^{-1}$ where $f(t)=(t+1)^\Delta-1$ for some $0<\Delta<\frac{1}{3}$.  To bound $\int_{\ii N^\e}^{\ii\infty}\exp[\frac N2(\Gmgn(\lambda_1+uN^{-1})-\Gmgn(\cpmgn))]\dd u$ we observe that 
\beq\begin{split}
\bigg|\int_{\ii N^\e}^{\ii\infty}&\exp\left[\frac N2(\Gmgn(\lambda_1+uN^{-1})-\Gmgn(\cpmgn))\right]\dd u\bigg|\\
&=\int_{0}^{\infty}\bigg|(f'(t)+i)\exp\left[\frac N2(\Gmgn(\lambda_1+(-f(t)+i(N^\e+t))N^{-1})-\Gmgn(\cpmgn))\right]\bigg|\dd t\\
&\leq\int_{0}^{\infty}\bigg|-\frac{\Delta}{(t+1)^{1-\Delta}}+i\;\bigg|\cdot\bigg|\exp\left[\frac N2(\Gmgn(\lambda_1+(-f(t)+i(N^\e+t))N^{-1})-\Gmgn(\cpmgn))\right]\bigg|\dd t\\
&\leq\int_{0}^{\infty}\sqrt{2}\;\bigg|\exp\left[\frac N2(\Gmgn(\lambda_1+(-f(t)+i(N^\e+t))N^{-1})-\Gmgn(\cpmgn))\right]\bigg|\dd t
\end{split}\eeq
Thus, it suffices to show $\int_{0}^{\infty}|\exp\left[\frac N2(\Gmgn(\lambda_1+(-f(t)+i(N^\e+t))N^{-1})-\Gmgn(\cpmgn))\right]|\dd u=O(N^{-\e/3})$.  We use the notation $\Gmgn(z)=A(z)+B(z)$ where
\beq
A(z)=\beta z-\frac{1}{N}\sum_{j=1}^N\log(z-\lambda_j)\qquad B(z)=\frac{(H+\xi)^2\beta}{N^2}\sum_{j=1}^N\frac{n_j^2}{z-\lambda_j}
\eeq
We begin by noting that 
\beq\begin{split}
&\left|\int\exp\left[\frac{N}{2}\left(\Gmgn(z)-\Gmgn(\cpmgn)\right)\right]\dd z\right| \\
&\leq\int\left|\exp\left[\frac{N}{2}\left(A(z)-A(\cpmgn)\right)\right]\right|\;\cdot\;
\left|\exp\left[\frac{N}{2}\left(B(z)-B(\cpmgn)\right)\right]\right|\dd z
\end{split}\eeq
Therefore, in order to show that the integral on the tail has order $O(N^{-\e/3})$, it is enough two prove the following two things:
\begin{itemize}
\item $\int_0^\infty\left|\exp\left(\frac{N}{2}(A(\lambda_1-f(t)N^{-1}+i(N^\e+t)N^{-1})-A(\cpmgn))\right)\right|\dd t=O(N^{-\e/3})$ and
\item $\left|\exp\left(\frac{N}{2}(B(\lambda_1-f(t)N^{-1}+i(N^\e+t)N^{-1})-B(\cpmgn))\right)\right|$ is bounded for $t>0$.
\end{itemize}
The integral in the first bullet point can be rewritten as follows:
\beq\begin{split}
\int_0^\infty&\left|\exp\left(\frac{N}{2}(A(\lambda_1-f(t)N^{-1}+i(N^\e+t)N^{-1})-A(\cpmgn))\right)\right|\dd t\\
=&\int_0^\infty\exp\left(\frac{N\beta}{2}(\lambda_1-f(t)N^{-1}-\cpmgn)\right)
\cdot\left|\exp\left[-\frac{1}{2}\sum_{j=1}^N\log\left(\frac{\lambda_1-f(t)N^{-1}+i(N^\e+t)N^{-1}-\lambda_j}{\cpmgn-\lambda_j}\right)\right]\right|\dd t\\
=&\int_0^\infty\exp\left(-\frac{\beta(\so_\mgn +f(t))}{2}\right)
\cdot\left|\exp\left[-\frac{1}{2}\sum_{j=1}^N\log\left(\frac{\lambda_1-f(t)N^{-1}+i(N^\e+t)N^{-1}-\lambda_j}{\cpmgn-\lambda_j}\right)\right]\right|\dd t\\
=&\int_0^\infty\exp\left(-\frac{\beta(\so_\mgn +f(t))}{2}\right)
\cdot\exp\left[-\frac{1}{2}\sum_{j=1}^N\log\left|\frac{\lambda_1-f(t)N^{-1}+i(N^\e+t)N^{-1}-\lambda_j}{\cpmgn-\lambda_j}\right|\right]\dd t
\end{split}\eeq
We begin by showing that this integral restricted to the interval $[2(\cpmgn-\lambda_N)N,\infty)$ is of order $O(e^{-N/2})$ .  If we integrate over just the first factor in the expression above, we get
\beq
\int_{2(\cpmgn-\lambda_N)N}^\infty\exp\left(-\frac{\beta(\so_\mgn +f(t))}{2}\right)\dd t
=\exp\left(-\frac{\beta(\so_\mgn -1)}{2}\right)\int_{2(\cpmgn-\lambda_N)N}^\infty\exp\left(-\frac{(t+1)^{\Delta}}{2}\right)\dd t=O(\exp(-N^\Delta))
\eeq
In the last equality above, we used the fact (see Lemma \ref{lem:ext1/2cpmgn}) that $\so_\mgn>T$ on the event $\event$.  Since this integral converges, it suffices to show that $\exp\left[-\frac{1}{2}\sum_{j=1}^N\log\left|\frac{\lambda_1-f(t)N^{-1}+i(N^\e+t)N^{-1}-\lambda_j}{\cpmgn-\lambda_j}\right|\right]$ is of order $O(e^{-N/2})$ for $t\geq2(\cpmgn-\lambda_N)N$.
\beq\begin{split}
\exp&\left[-\frac{1}{2}\sum_{j=1}^N\log\left|\frac{\lambda_1-f(t)N^{-1}+i(N^\e+t)N^{-1}-\lambda_j}{\cpmgn-\lambda_j}\right|\right]\\
\leq&\exp\left[-\frac{1}{2}\sum_{j=1}^N\log\left|\frac{(N^\e+t)N^{-1}}{\cpmgn-\lambda_j}\right|\right]
\leq\exp\left[-\frac{1}{2}\sum_{j=1}^N\log\left|\frac{2(\cpmgn-\lambda_N)}{\cpmgn-\lambda_j}\right|\right]\\
\leq&\exp\left[-\frac{N}{2}\log(2)\right]<e^{-N/2}
\end{split}\eeq
Next, we show that the integral is of order $O(N^{-\e/3})$ on the interval $[0, 2(\cpmgn-\lambda_N)N]$.  For $t$ in this interval we have
\beq\begin{split}
\exp&\left[-\frac{1}{2}\sum_{j=1}^N\log\left|\frac{\lambda_1-f(t)N^{-1}+i(N^\e+t)N^{-1}-\lambda_j}{\cpmgn-\lambda_j}\right|\right]\\
&=\exp\left[-\frac{1}{2}\sum_{j=1}^N\log\left|\frac{a_1-a_j-f(t)N^{-1/3}+i(N^\e+t)N^{-1/3}}{\so_\mgn N^{-1/3}+a_1-a_j}\right|\right].
\end{split}\eeq
To obtain an upper bound for this quantity, we begin with the $j=1$ term and observe that 
\beq\label{eq:intapproxj=1}
\exp\left[-\frac{1}{2}\log\left|\frac{-f(t)+i(N^\e+t)}{\so_\mgn }\right|\right]
\leq\exp\left[-\frac{1}{2}\log\left|\frac{N^\e}{\so_\mgn }\right|\right]
\leq\left(\frac{\so_\mgn }{N^\e}\right)^{1/2}.
\eeq
For the summation of the $j\geq2$ terms, we get
\beq\begin{split}
\exp&\left[-\frac{1}{2}\sum_{j=2}^N\log\left|\frac{a_1-a_j-f(t)N^{-1/3}+i(N^\e+t)N^{-1/3}}{\so_\mgn N^{-1/3}+a_1-a_j}\right|\right]\\
&\leq\exp\left[-\frac{1}{4}\sum_{j=2}^N\log\left|\frac{(a_1-a_j-f(t)N^{-1/3})^2+(N^\e+t)^2N^{-2/3}}{(\so_\mgn N^{-1/3}+a_1-a_j)^2}\right|\right]\\
&\leq\exp\left[-\frac{1}{4}\sum_{j=2}^N\log\left|1+\frac{-2(f(t)+\so_\mgn )(a_1-a_j)N^{-1/3}-\so_\mgn ^2N^{-2/3}+(N^\e+t)^2N^{-2/3}}{(\so_\mgn N^{-1/3}+a_1-a_j)^2}\right|\right].
\end{split}\eeq
Dropping some smaller terms, we see that the last line of the inequality above has upper bound
\beq\begin{split}
\exp&\left[-\frac{1}{4}\sum_{j=2}^N\log\left|1-\frac{2(f(t)+\so_\mgn )(a_1-a_j)N^{-1/3}}{(\so_\mgn N^{-1/3}+a_1-a_j)^2}\right|\right]\\
&\leq\exp\left[-\frac{1}{4}\sum_{j=2}^N\log\left|1-\frac{2(f(t)+\so_\mgn )N^{-1/3}}{a_1-a_j}\right|\right]\\
&\leq\exp\left[\frac{1}{4}\sum_{j=2}^N\left(\frac{2(f(t)+\so_\mgn )N^{-1/3}}{a_1-a_j}+\left(\frac{2(f(t)+\so_\mgn )N^{-1/3}}{a_1-a_j}\right)^2\right)\right]\\
&=\exp\left[\frac{f(t)+\so_\mgn}{2}\sum_{j=2}^N\frac{N^{-1/3}}{a_1-a_j}+((f(t)+\so_\mgn) N^{-1/3})^2\sum_{j=2}^N\frac{1}{(a_1-a_j)^2}\right].
\end{split}\eeq
Next, using the properties of the event $\event$, we see that the last line above has upper bound
\beq
\exp\left[\frac{f(t)+\so_\mgn }{2}\left((1+O(N^{-\frac13+\ep}))+O(N^{-\frac23+\Delta+\ep})\right)\right]\\
=\exp\left[\frac{f(t)+\so_\mgn }{2}\left(1+O(N^{-\frac13+\ep})\right)\right].
\eeq
Combining this with the upper bound from the $j=1$ term in \eqref{eq:intapproxj=1}, we conclude that 
\beq
\exp\left[-\frac{1}{2}\sum_{j=1}^N\log\left|\frac{\lambda_1-f(t)N^{-1}+i(N^\e+t)N^{-1}-\lambda_j}{\cpmgn-\lambda_j}\right|\right]
\leq\left(\frac{\so_\mgn }{N^\e}\right)^{1/2}
\exp\left[\frac{f(t)+\so_\mgn }{2}\left(1+O(N^{-\frac13+\ep})\right)\right].
\eeq
Finally, plugging this back into the original integral, we get
\beq\begin{split}
\int_0^{2(\cpmgn-\lambda_N)N}&\left|\exp\left(\frac{N}{2}(A(\lambda_1-f(t)N^{-1}+i(N^\e+t)N^{-1})-A(\cpmgn))\right)\right|\dd t\\
\leq&\int_0^{2(\cpmgn-\lambda_N)N}\exp\left(-\frac{\beta(\so_\mgn +f(t))}{2}\right)
\cdot\left(\frac{\so_\mgn }{N^\e}\right)^{1/2}\exp\left[\frac{f(t)+\so_\mgn }{2}(1+O(N^{-\frac13+\ep}))\right] \dd t\\
=&\left(\frac{\so_\mgn }{N^\e}\right)^{1/2}\int_0^{2(\cpmgn-\lambda_N)N}
\exp\left[-\frac{\beta-1+O(N^{-\frac13+\ep})}{2}\cdot(f(t)+\so_\mgn )\right] \dd t\\
\end{split}\eeq
Since $\beta>1$, there exists some $C''>0$ such that the integral is bounded above by
\beq
\left(\frac{\so_\mgn }{N^\e}\right)^{1/2}\int_0^{2(\cpmgn-\lambda_N)N}
\exp\left[-C''((t+1)^\Delta-1)\right] \dd t
=O\left(\left(\frac{\so_\mgn }{N^\e}\right)^{1/2}\right)
=O\left(\left(\frac{\ep\log N }{N^\e}\right)^{1/2}\right)
=O(N^{-\e/3})
\eeq
Lastly, it remains to show that $\left|\exp\left(\frac{N}{2}(B(\lambda_1-f(t)N^{-1}+i(N^\e+t)N^{-1})-B(\cpmgn))\right)\right|$ is bounded and it suffices to show that $\re\left(\frac{N}{2}(B(\lambda_1-f(t)N^{-1}+i(N^\e+t)N^{-1})-B(\cpmgn))\right)$ is bounded above.
\beq\begin{split}
\re&\left[\frac{N}{2}(B(\lambda_1-f(t)N^{-1}+i(N^\e+t)N^{-1})-B(\cpmgn)]\right)\\
=&\re\left[\frac{N}{2}\cdot\frac{(H+\xi)^2\beta}{N^2}\sum_{j=1}^N\left(\frac{n_j^2}{\lambda_1-f(t)N^{-1}+i(N^\e+t)N^{-1}-\lambda_j}-\frac{n_j^2}{\cpmgn-\lambda_j}\right)\right]\\
\end{split}\eeq
We observe that the real part of the $j=1$ term in the summation is negative and, furthermore, $\sum_{j=2}^N(-\frac{n_j^2}{\cpmgn-\lambda_j})$ is negative.  Removing these terms, we see that the quantity above has upper bound
\beq\begin{split}
&\re\left[\frac{(H+\xi)^2\beta}{2N}\sum_{j=2}^N\frac{n_j^2}{\lambda_1-f(t)N^{-1}+i(N^\e+t)N^{-1}-\lambda_j}\right]\\
&=\frac{(H+\xi)^2\beta}{2N}\sum_{j=2}^N\frac{n_j^2(\lambda_1-f(t)N^{-1}-\lambda_j)}{(\lambda_1-f(t)N^{-1}-\lambda_j)^2+(N^\e+t)^2N^{-2}}.
\end{split}\eeq
Now consider two cases.  For $t<N$, the expression in the last line is bounded above by
\beq
\frac{(H+\xi)^2\beta}{2N}\sum_{j=2}^N\frac{n_j^2}{\lambda_1-f(t)N^{-1}-\lambda_j}=\frac{(H+\xi)^2\beta}{2}\sum_{j=2}^N\frac{n_j^2N^{-1/3}}{a_1-a_j-f(t)N^{-1/3}}
\eeq
This will be $O(1)$ because $\sum_{j=2}^N\frac{n_j^2N^{-1/3}}{a_1-a_j}=1+O(N^{-\frac13+\ep})$ on the event $\event$ and, for sufficiently small $\ep$, we have $f(t)N^{-1/3}<\frac12 (a_1-a_2)$ since $f(t)N^{-1/3}=O(N^{\Delta-\frac{1}{3}})$ where $\Delta<\frac{1}{3}$ and $a_1-a_2>N^{-\ep/3}$ on $\event$.

In the case where $t\geq N$, we instead use the upper bound
\beq
\frac{(H+\xi)^2\beta}{2N}\sum_{j=2}^N\frac{n_j^2(\lambda_1-\lambda_j)}{(N^\e+t)^2N^{-2}}
\leq\frac{(H+\xi)^2\beta}{2N}\sum_{j=2}^N4n_j^2
\eeq
Since $n_j^2$ are i.i.d. chi-squared random variables, this sum is $O(1)$ with overwhelming probability.
\end{proof}


\section{Applying this method to the overlap of two replicas}\label{sec:2ol}

Using a method similar to the proofs in Sections \ref{sec:ext} and \ref{sec:intapprox}, we can prove Theorem \ref{thm:rep} for $\ovl$, the overlap of two replicas (this is a rigorous re-formulation of Result 10.6 from 
\cite{BCLW}).  The generating function for the overlap with a replica involves a double integral rather than a single integral, but we can use the same contour as in Sections \ref{sec:ext} and \ref{sec:intapprox} for both integrals and then transform to polar coordinates in order to prove the desired decay properties outside a neighborhood of the critical point.  While our method works to prove this theorem, we do not provide the details here because it also follows from Theorem 2.14 of \cite{LS}, as we explain below.

\begin{theorem}\label{thm:rep} Given $T<1$ and $h=HN^{-1/2}$ for some some fixed $H\geq0$, we have the following asymptotic formula for the moment generating function of $\ovl$, the overlap with a replica.  This formula holds on the event $\event$ (which has probability at least $1-N^{-\ep/10}$) for any sufficiently small $\ep>0$ and $\xi=O(1)$. 
\beq
	\langle e^{ \xi \frac{\ovl}{1-T} } \rangle 
	=
	\frac{ \cosh\left(   \frac{2\sqrt{1-T}H |n_1| }{T} \right) e^{\xi} + e^{-\xi} }
	{\cosh\left( \frac{2\sqrt{1-T}H |n_1| }{T} \right) +1} +O(N^{-\frac{1}{21}+\frac\ep7})
\eeq
\end{theorem}
Note that the leading order term on the right hand side is the moment generating function of a shifted Bernoulli random variable that takes values $1$ and $-1$ with probability $P$ and $1-P$ respectively, where
\beq
	P= \frac{ \cosh\left(  \frac{2\sqrt{1-T}H |n_1| }{T} \right) }{\cosh\left(  \frac{2\sqrt{1-T}H |n_1| }{T}  \right) +1}  .
\eeq
Thus, for large $N$, we can conclude that, on the event $\event$, the overlap $\ovl$ behaves in its leading order like a shifted Bernoulli random variable.  This conclusion also follows from Theorem 2.14 of \cite{LS}, which states that, for sufficiently $\ep>0$, there exist $\ep_1>0$ such that with probability at least $1-N^{-\ep_1}$ and all $t>0$,
\beq
\left\langle\mathbf{1}_{\{|N^{-1}\sigma^{(1)}\cdot\sigma^{(2)}\mp(1-\beta^{-1})|\leq t\} }\right\rangle
=\frac12\pm\frac12 \tanh^2\left(\sqrt{v_1^2\theta(\beta-1)}\right)
+N^\ep O\left(t+N^{-2/3+\ep}t^{-2}+N^{-1/3}\right).
\eeq
While their theorem is formulated and proved in a different manner than Theorem \ref{thm:rep}, their result implies ours.

\end{document}